\documentclass[twoside,11pt,reqno]{amsart}

\usepackage{upref,amsxtra,amssymb,amscd}
\usepackage{varioref}
\usepackage{verbatim}
\usepackage{epsfig}
\usepackage{color}
\usepackage{url}

\usepackage{epic,eepic,eucal}
\usepackage{enumerate}

\def\g{          \mathfrak g}
\def\h{           \mathfrak h}

\def\Or{          \mathcal O}

\def\htop{              h_{top}}

\def\a{         \alpha}

\def\R{ {\bf R}   }

\newcommand{\NN}{{\mathbb N}}
\newcommand{\RR}{{\mathbb R}}

\newcommand{\TT}{{\mathbb T}}

\newcommand{\ZZ}{{\mathbb Z}}
\newcommand{\CC}{{\mathbb C}}

\newtheorem{theo}{\sc Theorem}[section]
\newtheorem{prop}[theo]{\sc Proposition}

\newtheorem{lemm}[theo]{\sc Lemma}
\newtheorem{coro}[theo]{\sc Corollary}

\theoremstyle{definition}

\theoremstyle{remark}

\newtheorem{rema}[theo]{\sc Remark}
\newtheorem{rems}[theo]{\sc Remarks}

\numberwithin{equation}{section}

\begin{document}
\title[Simultaneous dense and nondense]{Simultaneous dense and nondense orbits for commuting maps}
\author{Vitaly Bergelson}
\thanks{V.B acknowledges support received from the National Science Foundation via Grant DMS-1162073}
\address{V.B.: Department of Mathematics, Ohio State University, Columbus, Ohio 43210}
\email{vitaly@math.ohio-state.edu}
\author{Manfred Einsiedler}
\address{M.E.:  Departement Mathematik, ETH Z\"urich, R\"amistrasse 101, 8092, Z\"urich, Switzerland}
\email{manfred.einsiedler@math.ethz.ch}
\thanks{M.E. acknowledges support by the SNF (200021-152819).}
\author{Jimmy Tseng}
\address{J.T.:  School of Mathematics, University of Bristol, University Walk, Bristol, BS8 1TW UK}
\email{j.tseng@bristol.ac.uk}
\thanks{J.T. acknowledges the research leading to these results has received funding from the European Research Council under the European Union's Seventh Framework Programme (FP/2007-2013) / ERC Grant Agreement n. 291147.}

\begin{abstract} We show that, for two commuting automorphisms of the torus and for two elements of the 
	Cartan action on compact higher rank homogeneous spaces, many points have drastically different orbit structures for the two maps. 
	Specifically, using measure rigidity, we show
	that the set of points that have dense orbit under one map and nondense orbit under the second 
	has full Hausdorff dimension.
\end{abstract}

\maketitle
\section{Introduction}\label{secIntro}

Let $f: X \rightarrow X$ be a dynamical system on a set $X$ with a topology. If $f$ is an endomorphism, we will always only consider the forward orbits of points $x\in X$ and, if $f$ is an automorphism we will work with full two-sided orbits. Let us say that a point has \textit{dense orbit} if its orbit closure equals $X$ and \textit{nondense orbit} otherwise.  We will denote by~$D(f)$ the set of points with dense orbit and by~$ND(f)$ the set of points with nondense orbit.   A natural question one could ask about $f$ is what is the nature of the sets $D(f)$ and $ND(f)$?  One way to quantify this nature is by some notion of size such as density, cardinality, Hausdorff dimension, or measure (assuming that these notions apply to $X$).  Instances of this question have been answered for many ergodic dynamical systems.  The gist of these answers is that many ergodic systems that have some hyperbolic or expanding behavior have in addition to $D(f)$ having full measure also the property that $ND(f)$ is winning in the sense of Schmidt games, from which it follows that both $D(f)$ and $ND(f)$ have full Hausdorff dimension (see~\cite{AN, BFK, Da, Da2,Do2,Far, Kl2,KM2, MT, T4, Ur} for example. 

Now, given another dynamical system $\tilde{f}$ on $X$, the natural extension of our question is to the joint behavior of the pair of systems.  For these systems, ergodicity under both maps immediately implies that the set $D(f)\cap D(\tilde{f})$ of points with dense orbits under both maps is of full measure.  Similarly, winning immediately implies that the set $ND(f)\cap ND(\tilde{f})$ of points with nondense orbits under both maps is of full Hausdorff dimension.  Finally, there is the mixed case $D(f)\cap ND(\tilde{f})$, which is not amenable to either the full-measure or winning techniques (as the intersection will be neither).  
The subject of this paper is this mixed case, where we will restrict ourselves to the case 
where $f$ and $\tilde{f}$ commute with each other.  We say a set $A \subset X$ is \textit{jointly invariant under $f$ and $\tilde{f}$} if $f(A) \subset A$ and $\tilde{f}(A) \subset A$.  

It is clear that $D(f)\cap ND(\tilde{f})$ could be empty.  For example, this could occur if there exists a nontrivial topological factor of~$X$ on which~$f$ and $\tilde{f}$ are equal or just too closely related. Therefore, the mixed case requires more assumptions.  We consider two types of systems.
%For this we will need some terminology.  

The first type of systems we consider is the case of commuting automorphisms $T$ and $S$ of the $d$-dimensional torus $\TT^d$. Consider the $\ZZ^2$-action $\a$ on $\TT^d$ that is generated by~$T$ and $S$. This action is called \textit{irreducible} if there are no proper, infinite, closed subgroups which are jointly invariant under the action.  A $\ZZ^2$-action $\a'$ on a torus $\TT^{d'}$ is an \textit{algebraic factor} of $\a$ if there is a surjective toral homomorphism $h:\TT^d \rightarrow \TT^{d'}$ such that $\a' \circ h = h \circ \a$ and is a \textit{rank-one factor} if, in addition, $\a'(\ZZ^2)$ has a finite-index subgroup consisting of the powers of a single map.  An endomorphism or automorphism of~$\TT^d$ is \textit{hyperbolic} if its associated matrix has no eigenvalues on the unit circle in the complex plane.   Let $\dim(\cdot)$ denote Hausdorff dimension.

Our main results are the following.
\begin{theo}\label{thmCommToralMapsLarNondenDen}  Let $T,S$ be commuting automorphisms of the torus $\TT^d$ for $d \geq 2$ such that $T$ and $S$ 
	generate an algebraic $\ZZ^2$-action without rank-one factors.  Assume that~$T$ is hyperbolic.
	Then $ND(T) \cap D(S)$ has full Hausdorff dimension.
\end{theo}

\begin{coro}\label{coroCommToralMapsLarNondenDen}
Fix an algebraic~$\ZZ^2$-action on~$\TT^d$ without rank-one factors and let~$T$ be a hyperbolic element of the action.
Let $\{S_k\}$ be the family of all maps which meet the conditions on $S$ in Theorem~\ref{thmCommToralMapsLarNondenDen}.  Then $ND(T) \cap \bigcap_k D(S_k)$ has full Hausdorff dimension.
\end{coro}

 Recall that the unstable foliation of a hyperbolic toral endomorphism $T$ is comprised of a family of unstable manifolds of the same dimension $\geq 1$.  We denote this dimension by $\dim(W^u(T))$.

\begin{theo}\label{thmCommToralMapsLarNondenDen2}  
	Let $T,S$ be commuting hyperbolic epimorphisms of the torus $\TT^d$ 
	for $d \geq 1$ such that $T$ and $S$ 
	generate an algebraic action without rank-one factors.  Then 
	\[\dim\big(ND(T) \cap D(S)\big) \geq \dim(W^u(T))\geq1.\]
\end{theo}

 We also consider commuting partially hyperbolic maps on homogeneous spaces.

\begin{theo}\label{thmCommGroupMapsLarNondenDen}
	Let $G$ be a connected semi-simple Lie group  such that each simple normal subgroup of~$G$ has $\RR$-rank $\geq 2$.  
	Let $a_1, a_2$ belong to a maximal abelian $\RR$-diagonal subgroup of $G$ and assume that they correspond to linearly independent directions in every simple factor of the Lie algebra of~$G$.  Let $\Gamma \subset G$ be a cocompact lattice in $G$.  Let $\theta_1, \theta_2$ be the actions on $X=\Gamma \backslash G$ associated with $a_1, a_2$, respectively.  Then $ND(\theta_1) \cap D(\theta_2)$ has full Hausdorff dimension. 
\end{theo}

\begin{coro}\label{coroCommGroupMapsLarNondenDen}
	Let~$X,a_1$, and~$\theta_1$ be as in Theorem~\ref{thmCommGroupMapsLarNondenDen}. 
If, for a countable collection of maps $\{\theta_{2,k}\}$, each map meets the conditions on $\theta_2$ in Theorem~\ref{thmCommGroupMapsLarNondenDen}, then $ND(\theta_1) \cap \bigcap_k D(\theta_{2,k})$ has full Hausdorff dimension.
\end{coro}

\begin{rems}
Commuting maps $T$ and $S$ are called \textit{multiplicatively dependent} if there exist $(t,s) \in \ZZ^2 \backslash \{(0,0)\}$ such that $T^t = S^s$ and, otherwise, are \textit{multiplicatively independent}. If~$T$ and $S$ are multiplicatively dependent then we have~$ND(T)\cap D(S)=\emptyset$. If $T$ and $S$ generate a $\ZZ^2$-action on~$\TT^d$, then either it has no rank-one factors and Theorem~\ref{thmCommToralMapsLarNondenDen} (respectively, Theorem~\ref{thmCommToralMapsLarNondenDen2}) applies or it has a factor on which the maps are multiplicatively dependent. 

We note that our results can be viewed as weak versions of Host's theorem \cite{Host} for much more general dynamical systems. In fact, Host has shown in \cite{Host} that for a $\times 2$-invariant and ergodic probability measure $\mu$ on $\TT$ with positive entropy, $\mu$-a.e.~point is generic for the $\times 3$-map and the Lebesgue measure on~$\TT$. Our method of proof is similar to the theorem of Johnson and Rudolph \cite{Rudolph-Johnson} who proved an averaged version of Host's theorem and also to the argument in \cite{E-F-S} and \cite{Shi}. 

We note that the compactness assumption in Theorem~\ref{thmCommGroupMapsLarNondenDen} and Corollary~\ref{coroCommGroupMapsLarNondenDen}, while used in this paper via the variational principle and the Ledrappier-Young formula, may not be necessary. 

We also believe that~$ND(T_1)\cap ND(T_2)\cap D(S_1)\cap D(S_2)$ is non-empty
and maybe even be of full Hausdorff dimension for commuting maps~$T_1,T_2,S_1,S_2$ in general position to each other, 
but our method does not seem to extend to that case.

We make essential use of the assumption that the two maps commute. If they do not commute, completely different techniques are required 
(see~\cite{Alex-Beverly}). 
\end{rems}

\subsection{Organization of this paper}  The proofs of the toral case (Theorems~\ref{thmCommToralMapsLarNondenDen}~and~\ref{thmCommToralMapsLarNondenDen2}) and the Lie group case (Theorem~\ref{thmCommGroupMapsLarNondenDen}) involve two steps:  finding a certain measure and using that measure to derive the desired dimension result.  The first step is very similar for both cases and is presented in Sections~\ref{secAGHM} and~\ref{secEquidistUnderS}.  We present the toral case and describe the changes necessary for the Lie group case.  The second step is (slightly) different for the two cases and is presented in Section~\ref{secHDFN}.  For the toral case, we apply the Ledrappier-Young formula.  For the Lie group case the  central directions require extra care.

\section{Averaging gives Haar measure}\label{secAGHM}  The first steps of the proofs of Theorems~\ref{thmCommToralMapsLarNondenDen},~\ref{thmCommToralMapsLarNondenDen2}, and~\ref{thmCommGroupMapsLarNondenDen} are similar.  We start the proof with the case of a torus.

\subsection{Measure Rigidity for commuting toral endomorphisms}  
The key to this first step of our proof is the following measure rigidity result \cite[Theorem 1.3]{EL}.  
For this let us recall that a solenoid is a connected compact abelian group whose dual group has finite rank.
The reason why solenoids are important for us is that the invertible extension of a surjective endomorphism~$R:\TT^d\to\TT^d$
gives an automorphism~$\hat{R}:X\to X$ of a solenoid~$X$. In fact the invertible extension can be 
realized as the shift map on the solenoid
\[
 X=\bigl\{(x_n)\in(\TT^d)^{\ZZ}\mid T(x_n)=x_{n+1}\mbox{ for all }n\in\ZZ\bigr\}.
\]

\begin{theo}[Measure rigidity]\label{thm: ELmain}
Let $\alpha$ be a $\ZZ^d$-action ($d \geq 2$) by automorphisms of
a solenoid $X$.
Suppose $\alpha$ has no rank one factors, and let
$\mu$ be an $\alpha$-ergodic measure on $X$.
Then there exists a subgroup $\Lambda \subset\ZZ^d$ of finite index and a decomposition
$\mu=\frac{1}{M}(\mu_1+\ldots+\mu_M)$ of $\mu$ into mutually singular measures
with the following properties for every $i=1,\ldots,M$.
\begin{enumerate}
\item
Every measure $\mu_i$ is $\alpha_\Lambda$-ergodic,
where $\alpha_\Lambda$ is the restriction of $\alpha$ to $\Lambda$.
\item
There exists an $\alpha_\Lambda$-invariant closed subgroup $G_i$
such that $\mu_i$ is invariant under translation under elements in
$G_i$, i.e.\ $\mu_i(A)=\mu_i(A+g)$ for all $g \in G_i$ and every
measurable set $A$.
\item
For $n \in\ZZ^d$,
$\alpha(n)_*\mu_i=\mu_j$
for some $j$ and $\alpha(n)(G_i)=G_j$.
\item
Every measure $\mu_i$ induces a measure on the factor
$X/G_i$ with $h_{\mu_i}(\alpha(n)_{X/G_i})=0$
for any $n \in \Lambda$. (Here $\alpha_{X/G_i}$ denotes the
action induced on $X/G_i$).
\end{enumerate}
\end{theo}

\begin{rema} A complete proof of the special case of \cite[Theorem 1.3]{EL} for irreducible systems is given in the research announcement \cite{EL}.   The general case will appear in \cite{ELW}.
	
	The restriction to a finite index subgroup~$\Lambda$ of the acting group~$\ZZ^d$ may be necessary
	since there may exist a subtorus~$G_1<\TT^d$ that is invariant under~$\Lambda$ but not under~$\ZZ^d$.
\end{rema}

\begin{coro}[High entropy case]\label{coroMeasureRigidity} 
	Let $T,S$ be commuting hyperbolic epimorphisms of the torus $\TT^d$ for $d \geq 1$ or 
	commuting  automorphisms of the torus $\TT^d$ for $d \geq 3$ such that 
	the induced~$\ZZ^2$-action has no rank one factor. 
  If $\mu$ is a $T,S$-invariant ergodic probability measure on $\TT^d$ with $h_\mu(T)$ close enough to the topological entropy $h_{\textrm{top}}(T)$, then $\mu$ is the Haar measure.
\end{coro} 

\begin{proof}
	Let us first assume that~$T$ and~$S$ are invertible and~$d\geq 3$.	
	We define~$\kappa\in(0,1)$ by
	\[
	 \kappa=\frac{\max\sum_{\lambda}'\log^{+}(|\lambda|)}{\sum_{\lambda}\log^{+}(|\lambda|)},
	\]
	where the sum in the denominator runs over all the eigenvalues $\lambda$ of the matrix corresponding to~$T$ of absolute value
	bigger than one taking into
	account the algebraic multiplicity of each eigenvalue, and $\sum'$ denotes any sum 
	where one of the eigenvalues is dropped or one of the multiplicities is reduced.
	
	Now suppose~$\mu$ is a~$T,S$-invariant and ergodic probability measure on~$\TT^d$ with~$h_\mu(T)>\kappa h_{\textrm{top}}(T)$. We apply Theorem~\ref{thm: ELmain} to find~$\Lambda$ and~$G_1$.
As~$\Lambda$ has finite index in~$\ZZ^2$ there exists some~$n\geq 1$ such that~$(n,0)\in\Lambda$ and so~$T^n$ preserves~$G_1$ and  $h_{\mu_1}(T^n_{X/G_1})=0$. We claim that for~$\kappa\in(0,1)$ as above we must have~$G_1=\TT^d$ and so~$\mu$ is the Lebesgue measure on~$\TT^d$. 

In fact, if~$G_1<\TT^d$ is a proper subgroup, then by the Abramov-Rokhlin entropy addition formula
\[
h_{\mu_1}(T^n)=h_{\mu_1}(T^n_{X/G_1})+h_{\mu_1}(T^n|\mathcal{B}_{X/G_1}),
\]
where~$\mathcal{B}_{X/G_1}\subset\mathcal{B}_{\TT^d}$ denotes the Borel sub-$\sigma$-algebra corresponding to the factor~$X/G_1$. Now we apply the standard inequality of entropy (see for instance \cite[Thm.~7.9]{Pisa}) for the relative entropy  $h_{\mu_1}(T^n|\mathcal{B}_{X/G_1})$ which gives
\[
 h_{\mu_1}(T^n|\mathcal{B}_{X/G_1})\leq \sum_{\zeta}\log^+(|\zeta|)
\]
where the sum goes over all the eigenvalues $\zeta$ of the matrix corresponding to~$T^n$ when restricted to the rational subspace corresponding to~$G_1$ with algebraic multiplicity. While $G_1$ may not be invariant under the
matrix corresponding to $T$, it is clear that the eigenvalues $\zeta$ of the matrix corresponding to $T^n$ when restricted to $G_1$ are $n$-th powers of the eigenvalues $\lambda$ of the matrix corresponding to $T$. Moreover, since
$G_1$ is a proper subtorus 
either the former set of eigenvalues is a proper subset or for one of the eigenvalues the corresponding multiplicities disagree. Therefore,
\[
 h_\mu(T)=\frac{1}n h_\mu(T^n)=\frac1nh_{\mu_1}(T^n|\mathcal{B}_{X/G_1})\leq \frac1n
 \sum_{\zeta}\log^+(|\zeta|)\leq\kappa\sum_\lambda\log^+(|\lambda|)=\kappa h_{\textrm{top}}(T),
\]
which contradicts our assumption on~$\mu$.

If either $T$ or $S$ is not invertible, then we construct the invertible extension $X$ of $\TT^d$ (e.g.\ using the
map~$R=ST$) and apply
Theorem~\ref{thm: ELmain} in the same way to this solenoid. 
\end{proof}

To use Corollary~\ref{coroMeasureRigidity}, we must relate topological entropy to dimension. 
For hyperbolic toral endomorphisms or automorphisms, all directions are either expanding or contracting (i.e. the mapping is Anosov).  Only the expanding directions contribute to the entropy.  Let $|\lambda_1| >1$ be the largest absolute value of an eigenvalue of the matrix corresponding to~$T$.   Let $X=\TT^d$.

\begin{prop}\label{propLowerEntropyBnd} Let $F \subset X$ be a closed, $T$-invariant set.  Then we have that \[h_{\textrm{top}}(T\mid_F) \geq h_{\textrm{top}}(T) - (d - \dim F) \log(|\lambda_1|).\] 
\end{prop}

\subsection{Proof of Proposition~\ref{propLowerEntropyBnd}}  \label{subsecProofthmLowerEntropyBnd}   We follow the standard proof for computing the topological entropy of toral endomorphisms (see, for example, \cite[Proposition~2.6.2]{BS}), but with changes to accommodate the set $F$.  Let $t_b$ be the lower box dimension of $F$.  Then $\RR^d= \bigoplus V_\lambda  \oplus \bigoplus V_{\lambda, \bar{\lambda}}$ where $V_\lambda$ is the generalized\footnote{If the eigenvalue $\lambda$ is real then the generalized eigen\-space is simply the maximal subspace on which~$\lambda$ is the only eigenvalue --- allowing Jordan blocks. If the eigenvalue $\lambda$ is complex we take the corresponding sum of the generalized eigenspaces $V_\lambda+V_{\bar{\lambda}}\subseteq\CC^d$ and intersect it with~$\RR^d$.} eigenspace for $\lambda \in \RR$ and $V_{\lambda, \bar{\lambda}}$ for $\lambda \in \CC \backslash \RR$.  Since there is no need to distinguish between generalized eigenspaces for real and non-real eigenvalues, we index them as follows:  \[\RR^d = \bigoplus_{i=1}^m V_i,\] where the indices $1, \cdots, k$ correspond to the expanding generalized eigenspaces and $k+1, \cdots, m$ to the contracting such that the corresponding eigenvalues for each generalized eigenspace are ordered \begin{eqnarray}\label{eqnOrderEigenvalues}|\lambda_1| \geq \cdots \geq |\lambda_k| > 1 >|\lambda_{k+1}| \geq \cdots \geq |\lambda_m| >0.\end{eqnarray}  Let $d_i = \dim(V_i)$ and let $d_{\textrm{con}}$ denote the sum of the dimensions of the contracting generalized eigenspaces---note that $d_{\textrm{con}}$ may be equal to $0$.  We use $T$ to denote both the linear map on $\RR^d$ and its induced map on $\TT^d$, relying on context to distinguish the two maps.

For each generalized eigenspace $V_i$, pick an orthonormal basis and impose the sup norm  $\|\cdot\|_i$. The metric \[d(v,w) := \max(\|v_1-w_1\|_1, \cdots, \|v_m-w_m\|_m)\] for vectors $v = v_1 + \cdots + v_m$ and $w = w_1 + \cdots + w_m$ (where $v_i, w_i \in V_i$) is invariant under translations and hence induces a metric on $\TT^d$, which is also denoted by $d$ and is given by considering the distance between cosets in $\RR^d$.  To compute topological entropy, we also define \[d_n(v,w) := \max \{d(f^j(v), f^j(w)) \mid 0\leq j\leq n-1\}.\]

For one-dimensional $V_i$, the corresponding eigenvalues are real and $T (v_i) = \lambda v_i$ is contraction or expansion by $\lambda$.  For the other types of $V_i$, this type of behavior, roughly speaking, also holds (\cite[Lemma 2.6.3]{BS}):

\begin{lemm}\label{lemmAlmostUniExpansion} Let $\lambda$ be the eigenvalue for $V_i$.  Then for every $\tilde{\delta} >0$ there is a $C(\tilde{\delta})\geq 1$ such that \[C^{-1} (|\lambda| - \tilde{\delta})^n \|v\|_i \leq \|T^n(v)\|_i \leq C(|\lambda| + \tilde{\delta})^n \|v\|_i\] for every $n \in \NN$, every generalized eigenspace $V_i$, and every $v \in V_i$.
\end{lemm}

We need to pick a small enough positive $\tilde{\delta}$ to separate eigenvalues.  Choose %\[\tilde{\delta}:= \min{\Bigg{\{}\bigg{\{}\frac{|\lambda_i|-|\lambda_{i+1}|}{4} \;\bigg{\vert}\; 1 \leq i \leq \tilde{k}-1 \bigg{\}} \bigcup \bigg{\{}\frac{|\lambda_{\tilde{k}}|-1}{4}, \frac{1-|\lambda_{{\tilde{k}}+1}|}{4} \bigg{\}}\Bigg{\}}}.\]  
\[\begin{array}{lr} 0 < \tilde{\delta}\leq \min\big{\{}\frac{|\lambda_k|-1}{4}, \frac{1-|\lambda_{k+1}|}{4} \big{\}} & \quad \textrm{ if } k < m, \\
 0 < \tilde{\delta}\leq \frac{|\lambda_k|-1}{4} & \quad \textrm{ if } k = m. 
\end{array}\]

%**********  May not be needed *****************
%Let $C$ be as in the lemma.  Then there exists an $N_0 \in \NN$ such that, for all $n \geq N_0$, \[C^{-1} (|\lambda_1|-\tilde{\delta})^n > C(|\lambda_2|+\tilde{\delta})^n.\]  Henceforth, we consider consider $n \geq N_0 + 1$.
%************ End May not be needed ********************

%Choose a small $\varepsilon >0$ and define $\delta := |\lambda_1|^{-n}\varepsilon$.  Let 
%\begin{itemize}
%	\item $N_F(\delta)$ be the minimal number of balls of radius $\delta$ needed to cover $F$ in the $d$ metric,
%	\item $\widetilde{Cov}_F(\varepsilon)$ be the minimal number of balls of radius $\varepsilon$ needed to cover $F$ in the $d_n$ metric, and 
%	\item $Cov_F(\varepsilon)$ be the minimal number of balls of radius $\leq \varepsilon$ needed to cover $F$ in the $d_n$ metric
%\end{itemize}

The union of the orthonormal bases of all generalized eigenspaces is a basis for $\RR^d$.  Fix this basis.  A \textit{parallelepiped $P$ (with respect to this basis)} is a $d$-dimensional closed parallelepiped with edges parallel to the basis elements.  The \textit{center} of the parallelepiped $P$ is the point in $\RR^d$ that is the barycenter of $P$.  The next lemma tells us that a ball in the $d_n$ metric is, roughly, a parallelepiped in the $d$ metric.

\begin{lemm}\label{lemmCompareEntMetrictoIndMetric} 
	Let $n\geq 0$. Let $A$ denote a closed ball in the $d_n$ metric of radius $\varepsilon>0$ around the point $\boldsymbol{0}$ and $B:=B(\varepsilon,n)$ denote the closed parallelepiped around center $\boldsymbol{0}$ whose edges parallel to the basis vectors of $V_i$ are all equal in length (with respect to the $d$ metric)  to 
\begin{itemize}
	\item  $2 C (|\lambda_i| - \tilde{\delta})^{1-n} \varepsilon$ for $1 \leq i \leq k$ (i.e. the expanding eigenspaces) and
	\item $ 2 C\varepsilon$ for $k+1 \leq i \leq m$ (i.e. the contracting eigenspaces, if present).
\end{itemize}  Then $A \subset B.$
\end{lemm}

\begin{proof}
Since the action of $T$ respects the splitting into these generalized eigenspaces, we may consider each such eigenspace separately.  Let $v \in A \cap V_i$.  Then \[\max\bigg{\{} \|T^jv\|_i \;\bigg{\vert}\; 0 \leq j \leq n-1\bigg{\}} \leq \varepsilon.\] There are two cases:  the expanding and the contracting.   

Let $V_i$ be an expanding generalized eigenspace.  By Lemma~\ref{lemmAlmostUniExpansion}, we have \[C^{-1} (|\lambda_i| - \tilde{\delta})^{n-1} \|v\|_i \leq \varepsilon\] because $|\lambda_i| - \tilde{\delta} > 1$.

Now let $V_i$ be a contracting generalized eigenspace.  Setting $j=0$ we see
that
 \[ \|v\|_i \leq \varepsilon\leq C\varepsilon.\]

Applying the length constraints on $B$ gives the desired result.\end{proof} 

Choose a small $\varepsilon >0$ and define $\delta_n := C (|\lambda_1| - \tilde{\delta})^{1-n}\varepsilon$.  (Recall that $\lambda_1$ is the eigenvalue with largest absolute value.)  Let 
\begin{itemize}
	\item $N_F(\delta_n)$ be the minimal number of balls of radius $\delta_n$ needed to cover $F$ in the $d$ metric,
	\item $\widetilde{Cov}'_F(\varepsilon,n)$ be the minimal number of translated parallelepipeds of the same orientation and side lengths as $B(\varepsilon,n)$ 
	from Lemma~\ref{lemmCompareEntMetrictoIndMetric} needed to cover $F$,
	\item $\widetilde{Cov}_F(\varepsilon,n)$ be the minimal number of balls of radius $\varepsilon$ needed to cover $F$ in the $d_n$ metric, and 
	\item $Cov_F(\varepsilon,n)$ be the minimal number of sets, contained in $F$, of diameter $\leq \varepsilon$ needed to cover $F$ in the $d_n$ metric
\end{itemize}

Lemma~\ref{lemmCompareEntMetrictoIndMetric} has the following immediate corollary:

\begin{coro} \label{coroCovNest1}
$\widetilde{Cov}_F(\varepsilon,n) \geq \widetilde{Cov}'_F(\varepsilon,n).$
\end{coro}

%We must also cover the elements of $\widetilde{Cov}_F(\varepsilon)$:
Likewise, we have

\begin{lemm} \label{lemmCovNest} $Cov_F(\varepsilon,n) \geq \widetilde{Cov}_F(\varepsilon,n).$
\end{lemm}
\begin{proof}
Any subset of $F$ with diameter (in the $d_n$ metric) $\leq \varepsilon$ is contained in a closed ball of $\TT^d$ (in the $d_n$ metric) of radius $\varepsilon$.  Take a covering of $F$ with cardinality $Cov_F(\varepsilon)$ and put each element of this covering into a closed ball of radius $\varepsilon$.  Hence we obtain a covering by closed balls and the result is now immediate.
\end{proof}

Let $P, Q$ be (closed) parallelepipeds with respect to the basis.  Let the edges of $P$ be integer multiples lengths (in the $d$ metric) of the respective edges of $Q$ --- let $v_j$ be a basis vector and $\ell_j$ be the ratio of the side lengths of $P$ and $Q$ in the direction of $v_j$.  Then a \textit{tiling} of $P$ by $Q$ is a finite collection of translates of $Q$, $\{Q + \boldsymbol{v}:\boldsymbol{v}\in I\}$, such that 
\begin{enumerate}
	\item \[P = \bigcup_{\boldsymbol{v}\in I} Q + \boldsymbol{v} \quad \textrm{ and }\]
	\item \[(Q + \boldsymbol{v}) \cap (Q+ \boldsymbol{v'}) \quad \textrm{ is for any pair $ \boldsymbol{v}, \boldsymbol{v'}$ either empty or a complete face}.\]
\end{enumerate}  The cardinality of the tiling is $\prod_{j=1}^d \ell_j$.  If the integer multiple condition no longer holds for all edges, then one can generalize the notion of tiling as follows.  Given $P$ as above, and a parallelepiped $Q$ with respect to the basis, which has a translate $Q+\boldsymbol{w}$ contained in $P$, let \[\ell_j := \frac {\textrm{length of an edge of $P$ parallel to the basis vector } v_j} {\textrm{length of an edge of $Q$ parallel to the basis vector } v_j} \geq 1.\]  (Here both lengths are with respect to the $d$ metric.)  Then a \textit{tiling} of $P$ by $Q$ is a collection of translates of $Q$ with cardinality $\prod_{j=1}^d \lceil \ell_j \rceil$, $\{Q + \boldsymbol{v}\}$, such that condition (2) above holds and condition (1) is replaced with 
\[
P \subset \bigcup_{\boldsymbol{v}} Q + \boldsymbol{v}.
\]
% Note that the removal of any element of this tiling results in a violation of condition (1).  A tiling of $P$ by $Q$ is not unique unless the $\ell_j$ are all integers. 
%For every under-tiling of $P$ by $Q$---denoted by $\mP_Q$, we define an \textit{over-tiling} of $P$ by $Q$---denoted by $\mP^Q$---such that $\mP_Q \subset \mP^Q$ and $\mP^Q \backslash \mP_Q$ is a collection of cardinality $\prod_{j=1}^d \lceil \ell_j \rceil - \prod_{j=1}^d \lfloor \ell_j \rfloor$ comprised of translates of $\{Q + \boldsymbol{v}\}$ such that condition (2) above holds for all elements of $\mP^Q$ and condition (1) is replaced with 
%\begin{itemize}
%	\item \[P \subset \bigcup_{\boldsymbol{v}} Q + \boldsymbol{v}.\]
%\end{itemize}.
 %Partition (roughly) these parallelepipeds into balls in the $d$ metric:

Recall that $d_i = \dim(V_i)$ and $d_{\textrm{con}}$ is the sum of the dimensions of the contracting generalized eigenspaces. With these notions we obtain the following:

\begin{lemm} \[\Big\lceil(|\lambda_1| - \tilde{\delta})^{(n-1)}\Big\rceil^{d_{\textrm{con}}} \prod_{i = 1}^k \Bigg\lceil \Bigg(\frac{|\lambda_1| - \tilde{\delta}}{|\lambda_i| - \tilde{\delta}}\Bigg)^{(n-1)}\Bigg\rceil^{d_i} \widetilde{Cov}'_F(\varepsilon,n) \geq  N_F(\delta_n).\]
\end{lemm}

\begin{proof}
Take a covering corresponding to $\widetilde{Cov}'_F(\varepsilon,n)$---this is a covering by parallelepipeds with side lengths given by the formulas in Lemma \ref{lemmCompareEntMetrictoIndMetric}.  A ball in the $d$ metric is also a parallelepiped with respect to the basis.  Pick an element $P$ of the covering.  Tile $P$ using such balls of radius $\delta_n$.  

The cardinality of this tiling is \[\Big\lceil(|\lambda_1| - \tilde{\delta})^{(n-1)}\Big\rceil^{d_{\textrm{con}}} \prod_{i = 1}^k \Bigg\lceil \Bigg(\frac{|\lambda_1| - \tilde{\delta}}{|\lambda_i| - \tilde{\delta}}\Bigg)^{(n-1)}\Bigg\rceil^{d_i}.\]

Tiling the other elements of the covering $\widetilde{Cov}'_F(\varepsilon,n)$ in the analogous way yields a covering of $F$ by balls (in the $d$ metric) of radius $\delta_n$.  The desired result is now immediate.
\end{proof}

The lemma has an immediate corollary:

\begin{coro}\label{coroCovNest2}  %For $n$ large enough (depending on $\lambda_1$, $\lambda_2$, and $\tilde{\delta}$), we have
\[2^d(|\lambda_1| - \tilde{\delta})^{d(n-1)} \prod_{i = 1}^k \frac 1 {(|\lambda_i| - \tilde{\delta})^{d_i(n-1)}} \widetilde{Cov}'_F(\varepsilon,n) \geq  N_F(\delta_n).\]
\end{coro}

Since the $d$ metric is induced by a norm and since all norms are equivalent on $\RR^d$, we have  by the definition of lower box dimension $t_b$ the following inequality for all sufficiently big~$n$:

\[\log(N_F(\delta_n)) \geq (t_b-\tilde\delta) \log(\frac 1 \delta_n).\] 

Applying Corollaries~\ref{coroCovNest1}~and~\ref{coroCovNest2} and Lemma~\ref{lemmCovNest} yields \begin{align*}Cov_F(\varepsilon,n) & \geq \frac 1 {\delta_n^{t_b-\tilde\delta}} 2^{-d}(|\lambda_1| - \tilde{\delta})^{-d(n-1)} \prod_{i = 1}^k (|\lambda_i| - \tilde{\delta})^{d_i(n-1)} \\ & \geq (C\varepsilon)^{-t_b+\tilde\delta} 2^{-d}(|\lambda_1| - \tilde{\delta})^{(t_b-d-\tilde\delta)(n-1)} \prod_{i = 1}^k (|\lambda_i| - \tilde{\delta})^{d_i(n-1)}.
\end{align*}

Applying the definition of topological entropy yields\begin{align*} h_{\textrm{top}}(T\mid_F) & = \lim_{\varepsilon \rightarrow 0^+} \limsup_{n \rightarrow \infty}\frac 1 n \log(Cov_F(\varepsilon,n))\\ & \geq (t_b-d-\tilde\delta) \log(|\lambda_1| - \tilde{\delta}) + \sum_{i=1}^k d_i \log(|\lambda_i| - \tilde{\delta}).
\end{align*}

Since this calculation holds for all $\tilde{\delta}$ small enough, we have that \begin{align*}h_{\textrm{top}}(T\mid_F) & \geq (t_b-d) \log(|\lambda_1|) + \sum_{i=1}^k d_i \log(|\lambda_i|) \\ & = h_{\textrm{top}}(T) - (d-t_b) \log(|\lambda_1|) .
\end{align*}

Since lower box dimension is greater than or equal to Hausdorff dimension, we have shown Proposition~\ref{propLowerEntropyBnd}.

\subsection{Averaging measures}\label{subsecAverMeas}

Let $X$ be either $\TT^d$ or $\Gamma \backslash G$.  A subset of $X$ is called {\em thick} (\cite{KW2}) if its intersection with any nonempty open set of $X$ has Hausdorff dimension equal to that of $X$ itself and called {\em winning} if it is winning in the sense of Schmidt games~\cite{Sch2} (or in the sense of the variations on Schmidt games~\cite{KW} and~\cite{Mc}).  The property of being thick is implied by (any of the variations on) winning and is the property that concerns us.\footnote{Since we only need to be aware of two properties of winning sets (that they have full Hausdorff dimension and that the winning property is preserved under taking countable intersections), we omit the definition, which was introduced in~\cite{Sch2} with later adaptations in~\cite{Mc} (and others).  A convenient summary of the theory of winning sets can be found in~\cite[Section~2.1]{ET}.}  In particular, the set of points with nondense orbits is thick (see~\cite{Da2}, \cite{T4}, and~\cite{BFK} for the toral case and~\cite{KM2} and~\cite{KW2}  for the Lie group case).

Let $T,S$ be the two commuting actions on $X$ and $x_0$ be a point of $X$.  Fix a sequence of open balls $U_q$ centered at $x_0$ and whose radius $\rightarrow 0$ as $q \rightarrow \infty$.  Define \[E(q) := E_{T, x_0}(q) := \{x \in X \mid \Or_T(x) \cap U_q = \emptyset\}\] where $\Or_T(x)$ denotes either the forward orbit for $T$ an endomorphism or the full orbit for $T$ an automorphism.  Note that $E(q)$ is a closed $T$-invariant set.  And the union $\cup_q E(q)$ is the subset of points whose orbit closures do not contain $x_0$.  The union has large Hausdorff dimension:

\begin{prop} The set $\bigcup_q E(q)$ is thick and, for the case of the torus, winning and, therefore,  $\bigcup_q E(q)$ has full Haudorff dimension.
\end{prop}
\begin{proof}
Apply~\cite[Theorem~1.1]{BFK} and~\cite[Theorem~1.1]{KM2}.
\end{proof}

 The proposition implies that $\dim (E(q)) \rightarrow \dim (X)$ as $q \rightarrow \infty$.  Thus, applying also Proposition~\ref{propLowerEntropyBnd} shows that $\htop(T\mid_{E(q)}) \approx \htop(T)$ for $q$ large enough.  Next, an application of the variational principle shows that there exists a Borel probability measure $\nu$ (whose support lies in $E(q)$) invariant under $T$ such that $h_{\nu}(T \mid_{E(q)})$ is as close to $\htop(T)$ as we like, provided that we choose $q$ large enough.  %note that $\nu$ has to be nonatomic because of the invariance, the fact that our various $X$s are infinite sets, and the fact that we can omit all eventually periodic points and still not affect the Hausdorff dimension. 

Let $\mu$ be the weak-$*$ limit along a subsequence of $N$s of the averaging measures\footnote{As an analogue to the result in~\cite{Rudolph-Johnson} we believe that the full sequence actually converge to the Haar measure but we neither need this statement nor do we have a proof.}
:\[\mu_N:= \frac 1 N \sum_{n=0}^{N-1} S_*^n \nu.\] Any such measure $\mu$ is $S$-invariant.

\begin{lemm}\label{lemm-T-inv}  The measure $\mu$ is $T$-invariant.
\end{lemm}
\begin{proof}
By commutativity, the measures $S_*^n \nu$ are $T$-invariant, the same holds for the convex
combination~$\mu_N$ and therefore also for the limit $\mu$.
\end{proof}

\begin{lemm}\label{lemmTopMeasEntClose} By choosing $q$ large enough, we can have $h_\mu(T)$ as close to $\htop(T)$ as we like.
\end{lemm}

\begin{proof}
	As is well known, a factor map may only decrease entropy. This shows that $h_{S^n\nu}(T)\leq h_\nu(T)$. On the other hand, $S^n$ is a finite-to-one factor map from which one sees that in fact $h_{S^n\nu}(T)=h_\nu(T)$. Due to the convexity of entropy this implies that $h_{\mu_N}(T)=h_\nu(T)$. Finally, it follows by the upper semicontinuity of measure-theoretic entropy that $h_\mu (T)\geq h_\nu(T)$, and, consequently, we can have $h_\mu(T)$ as close to $\htop(T)$ as we like.
\end{proof}

Let $m_X$ denote the probability Haar measure on $X$.  

\begin{prop}\label{PropErgCompHaar}The probability Haar measure $m_X$ is an ergodic component of $\mu$
	(of positive proportion).
\end{prop}
\begin{proof}
Let $\mathcal{E}$ be the $\sigma$-algebra of jointly $T$- and $ S$-invariant Borel subsets of $X$. 
The ergodic decomposition of $\mu$ with respect to the joint action can be obtained for instance via the decomposition of measures
\[\mu=\int_X\mu_x^{\mathcal{E}}d\mu(x)\] for the $\sigma$-algebra $\mathcal{E}$. Convexity of entropy gives
\[h_\mu(T)=\int_X h_{\mu_x^{\mathcal{E}}}d\mu(x).\] This shows that for a positive proportion of $x\in X$
we must have $h_{\mu_x^{\mathcal{E}}}(T)\geq h_\mu(T)$. If $h_\mu(T)$ is sufficiently close to $h_{\textrm{top}}(T)$,
we may apply measure rigidity (i.e. Corollary~\ref{coroMeasureRigidity}) for each such~$\mu_x^\mathcal{E}$ and obtain the proposition.
\end{proof}

\subsection{The Lie group case}\label{subsecLieGroupCase}  We now show how to derive the same results for the Lie group case.  Let $G, \Gamma,$ and $X$ be as in Theorem \ref{thmCommGroupMapsLarNondenDen}. Explicitly,  $G$ is a semisimple, real Lie group such that each factor has $\RR$-rank $\geq 2$, and $\Gamma$ is a lattice of $G$. 
Assume furthermore that $\a:\ZZ^2\rightarrow G$ is the parametrization of
a subgroup of a maximal Cartan subgroup of $G$ that projects injectively and discretely to each simple factor. We will identify $\alpha$ also with the induced action on the right of $X = \Gamma \backslash G$. We write $\a^{\boldsymbol{t}}$ for the action of an individual 
element of a two-dimensional subgroup of the Cartan subgroup where $\boldsymbol{t} \in \ZZ^2$. In this situation we can 
use~\cite[Theorem 2.4]{EK} which contains in particular the following result.

\begin{theo}[Measure rigidity] \label{theoMeasureRigidityII} Let $G$, $\Gamma$, and $\a$ be as above and 
	$\boldsymbol{t}\neq 0$ be a fixed element.  If $\mu$ is an $\a$-invariant ergodic probability measure on $\Gamma 
	\backslash G$ with $h_\mu(\a^{\boldsymbol{t}})$ close enough to $h_{\textrm{top}}(\a^{\boldsymbol{t}})$, then $\mu$ 
	is the Haar measure on $X$.
\end{theo}

Let $\tilde{\lambda}$ be a nonzero root of $G$ with respect to the Cartan subgroup that contains the image of $\a$.  The root $\tilde{\lambda}$ can be expressed as a linear map on the Lie algebra of the Cartan subgroup which we may identify with the Cartan subgroup.  Let $\g_{\tilde{\lambda}}$ be the root space corresponding to $\tilde{\lambda}$.    Then the adjoint action of  $\a^{\boldsymbol{t}}$ on $\g_{\tilde{\lambda}}$ is multiplication by the eigenvalue $\lambda:=e^{\tilde{\lambda}({\boldsymbol{t}})}$.

Fixing the element $\a^{\boldsymbol{t}}$, we can order the absolute value of the eigenvalues from each of the roots to reproduce (\ref{eqnOrderEigenvalues})---which, recall, is counted with multiplicity.  For the Lie group $G$, the action is only partially hyperbolic---in particular, its Lie algebra is $\mathfrak {g} = \mathfrak {h}^+ \oplus \mathfrak{h}^0 \oplus \mathfrak{h}^-$, and, on those subgroups corresponding to the direct summands of $\g$, the action expands,  stays isometric, or contracts, respectively.  As in the toral case, only the expanding directions, which correspond to the roots for which $|\lambda|>1$, contribute to the entropy.  It is well-known that the topological entropy (or the entropy with respect to Haar measure) is the sum of the logarithms of the absolute values of these expanding eigenvalues counted with multiplicity according to the real dimension of their corresponding  eigenspaces for the adjoint representation:  \[h_{\textrm{top}}(\a^{\boldsymbol{t}})= \sum_{i=1}^k \log|\lambda_i| \dim(g_{{\lambda}_i})= \sum_{\tilde{\lambda}(\boldsymbol{t})>0} \tilde{\lambda}(\boldsymbol{t}) \dim(g_{\tilde{\lambda}}).\] 

Let $a_1, a_2, \theta_1, \theta_2$ be as in Theorem~\ref{thmCommGroupMapsLarNondenDen}.  We 
may define the subgroup~$\a$ in such a way that $\theta_1 = \a(\boldsymbol{e}_1)$ and $\theta_2 = \a(\boldsymbol{e}_2)$.  Using the exponential map, which is a local diffeomorphism and bi-Lipschitz in a neighborhood of the origin, we see that the proof of Proposition~\ref{propLowerEntropyBnd}  for the torus case is also valid for the Lie group case (with $T$ replaced by $\theta_1$). In fact, the proof of Proposition~\ref{propLowerEntropyBnd}
is slightly easier in the Lie group case as the elements of the Cartan subgroup are diagonalizable on the Lie
algebra with real eigenvalues, hence we may work with real eigenspaces instead of generalized eigenspaces as in the torus case. Now Section~\ref{subsecAverMeas} is also valid for the Lie group case, provided that one replaces $T$ by $\theta_1$ and $S$ by $\theta_2$ and uses the correct version of measure rigidity, namely Theorem~\ref{theoMeasureRigidityII}.

\begin{rema}
	 If~$G$ has precisely $\RR$-rank two, $\Gamma$ is irreducible, and $\alpha$ parametrizes the Cartan subgroup, then an analogue of Theorem~\ref{theoMeasureRigidityII} can be derived from \cite[Theorem 1.1 and Corollary 1.2]{EL2}. In fact, Theorem~\ref{theoMeasureRigidityII} should hold very generally, but, as an example, for a two-dimensional subgroup in general position of the Cartan subgroup of~$G=\operatorname{SL}_2(\R)^3$ this does not follow from \cite{EK} or \cite{EL2}. See also~\cite[Thm.~1.4]{EL3}, where a related high entropy theorem
	is proven under milder (but still not weakest possible) assumptions.
\end{rema}

%*** stop here need to define the directions --- these should not be elements of the lie algebra, but have to do with the action of $\a^{\boldsymbol{t}}$---in the torus case $\a^{\boldsymbol{t}}$ is some matrix, the directions are eigendirections---what about the Lie group case?*******

\section{Equidistribution of $E(q)$ under $S$}\label{secEquidistUnderS}

Let $X$ be either $\TT^d$ or $\Gamma \backslash G$. We will write $T$ and $S$ for the two commuting maps that we consider on $X$.  As before, let \[D(S)=\{x\in X\mid \overline{\Or_S(x)}=X\}\] where $\Or_S(x)$ denotes either the forward orbit for $S$ an endomorphism or the full orbit for $S$ an automorphism.

\begin{lemm}\label{lemmSDensePointsTInv}  The set $D(S)$ is
	$T$-invariant.
\end{lemm}

\begin{proof}
	If $x\notin D(S)$ so that $\overline{\Or_S(x)}\subset X$ is a closed proper $S$-invariant subset, then both $T(\overline{\Or_S(x)})$ and $T^{-1}(\overline{\Or_S(x)})$ are again closed proper $S$-invariant subsets. Therefore,
	$x\notin D(S)$ implies $T(x)\in D(S)$ and $T^{-1}(x)\in D(S)$ (resp.\ $T^{-1}(x)\subset D(S)$ if $T$ is not invertible). 
\end{proof}

We will write $d(\cdot,\cdot)$ for a metric on $X$.
Now let $\{x_i:i=1,\ldots\}$ be a dense subset of $X$. Using this set we will define in a moment a countable partition of \[ND(S)=X\setminus D(S),\] where we will use a total order of $\NN\times\NN$. We define $(i,n)<(i',n')$ if either $i+n<i'+n'$ or $i+n=i'+n'$ and $i<i'$. Now we define inductively
\begin{eqnarray*}
 ND(1,1)&=&\bigl\{x\in X\mid d(\overline{\Or_S(x)},x_1)\geq 1\bigr\},\\
 ND(i,n)&=&\bigl\{x\in X\mid d(\overline{\Or_S(x)},x_i)\geq\frac1n\bigr\}\setminus\bigcup_{(j,m)<(i,n)}ND(j,m)
\end{eqnarray*}
for all~$(i,n)\in\NN\times\NN$. It is clear that $ND(1,1)$ and $\bigcup_{(j,m)\leq (i,n)}ND(j,m)$  are closed sets,
and that~$ND(S)=\bigcup_{(i,n)}ND(i,n)$. 

Recall the definition of the proper $T$-invariant closed sets $E(q)$ from Section~\ref{subsecAverMeas}. Pick $q$ large so that $E(q)$ is a set with close to maximal dimension and close to maximal topological entropy for $T$. Using the variational principle, we choose a $T$-invariant measure on $E(q)$ of entropy close to the topological entropy of $T|_{E(q)}$. By convexity of measure-theoretic entropy and the ergodic decomposition, we may assume that $\nu$ is $T$-invariant and ergodic.  Specifically, we have the following corollary to the measure rigidity results.

\begin{theo}\label{theoNuAisFullMeas} 
	For large enough $q$, we have that $\nu(D(S)) =1$.  
\end{theo}

\subsection{Proof of Theorem~\ref{theoNuAisFullMeas}}
Recall that we assume (as we may) that~$\nu$ is ergodic w.r.t.~$T$ and that~$D(S)$ is~$T$-invariant by Lemma~\ref{lemmSDensePointsTInv}. 
We assume indirectly that $\nu(D(S))=0$. 
 Then we may decompose~$\nu$ into $\nu=\sum_{(i,n)}\nu_{(i,n)}$ where $\nu_{(i,n)}=\nu|_{ND(i,n)}$ for all $(i,n)\in\NN\times\NN$. Using Tychonoff-Alaoglu theorem
we may choose a subsequence $N_k$ of the integers such that for all $(i,n)$ the average
\[
\frac1{N_k} \sum_{m=0}^{N_k-1}S_*^m\nu_{(i,n)}
\]
converges in the weak$^*$ topology to an $S$-invariant measure $\mu_{(i,n)}$. In particular, we obtain that 
$\frac1{N_k}\sum_{m=0}^{N_k-1}S_*^m\nu$ converges to an $S$-invariant probability measure $\mu=\sum_{(i,n)}\mu_{(i,n)}$. Note that by Lemma \ref{lemm-T-inv} $\mu$ is also $T$-invariant (but the same may not be true for $\mu_{(i,n)}$). 

Fix some $(i,n)\in\NN\times\NN$. 
As $\nu_{(i,n)}$-a.e. point $x\in X$ satisfies $d(S^m(x),x_i)\geq\frac1n$ it follows that $S_*^m\nu_{(i,n)}$
gives zero mass to the $\frac1n$-ball $B_{\frac1n}^X(x_i)$ around $x_i$. This implies furthermore that $\mu_{(i,n)}(B_{\frac1n}^X(x_i))=0$. Since $\mu_{(i,n)}$ is an $S$-invariant measure we conclude that $\mu_{(i,n)}$
is singular to the Haar measure $m_X$ of $X$. 

Since $(i,n)$ was arbitrary we obtain that $\mu=\sum_{(i,n)}\mu_{(i,n)}$ is singular to the Haar measure $m_X$ of $X$. However, this contradicts Proposition~\ref{PropErgCompHaar} which said that $\mu$ has the Haar measure as an ergodic component for the joint action of $T$ and $S$ with positive proportion. This contradiction shows that we must have $\nu(D(S))=1$ and concludes the proof of Theorem~\ref{theoNuAisFullMeas}.

\section{Hausdorff dimension from $\nu$}\label{secHDFN} 

In this section, we derive our main results, Theorems~\ref{thmCommToralMapsLarNondenDen},~\ref{thmCommToralMapsLarNondenDen2}, and~\ref{thmCommGroupMapsLarNondenDen}.  %The proof of these two theorems are (slightly) different. 
% First, we present the toral case and then the Lie group case and finally make a comment about the Hausdorff dimension of the measure $\nu$.

\subsection{Proof of Theorems~\ref{thmCommToralMapsLarNondenDen}~and~\ref{thmCommToralMapsLarNondenDen2}}\label{subsecProofofMainThmforTorus}  There are two ingredients:  Theorem~\ref{theoNuAisFullMeas} and the formula relating entropy, dimension, and Lyapunov exponents developed by Ledrappier and Young in~\cite{LY1} and~\cite{LY2}.  First, we present the case of toral automorphisms.  Then we present the changes necessary for the case of toral endomorphisms.

\subsubsection{Toral automorphisms}  Pick $q$ large so that the ergodic measure $\nu$ on $E(q)$ from Theorem~\ref{theoNuAisFullMeas} has measure-theoretic entropy $h_{\nu}(T \mid_{E(q)})$ close to $\htop(T)$ (see Section~\ref{subsecAverMeas}) and the conclusion of the theorem applies.  We use the Ledrappier-Young formula (\cite{LY2}, Theorem C$'$):
\begin{theo}[Ledrappier-Young Formula] Let $f:M \rightarrow M$ be a $C^2$-diffeomorphism of a compact Riemannian manifold and let $m$ be an ergodic Borel probability measure on $M$.  Let $\kappa_1 > \cdots >\kappa_u$ denote the distinct positive Lyapunov exponents of $f$, and let $\delta_i$ be the dimension of $m$ on the $W^i$-manifolds.  Then for $1 \leq i \leq u$ there are numbers $\gamma_i$ with $0 \leq \gamma_i \leq \dim(E_i)$, such that \[\delta_i = \sum_{j \leq i} \gamma_j\] for $i = 1, \cdots, u$ and \[h_m(f) = \sum_{i \leq u} \kappa_i \gamma_i.\]
\end{theo}

 Here the $W^i$ denotes the $i$-th unstable manifold for the dynamical system (which is  tangent to~$\bigoplus_{j\leq i}E_j$). 
In particular, $W^u$ is the unstable manifold.  The full set of Lyapunov exponents (positive, zero, and negative) gives rise to the corresponding decomposition of the tangent space at $x$, \[E_1(x) \oplus \cdots \oplus E_r(x).\]  

In our case
these subspaces of course simply correspond to the generalized eigenspaces. 
Moreover, the set of Lyapunov exponents (for $T$) is $\{\log|\lambda_i|\}$ where the $\lambda_i$ are the eigenvalues from Section~\ref{subsecProofthmLowerEntropyBnd}---note that the indexing of the Lyapunov exponents, the $\kappa_j$, and the indexing of the eigenvalues, the $\lambda_i$, may not match because the Lyapunov exponents are distinct, while the eigenvalues can contain duplicates.

As previously noted, we have that $\htop(T) = \sum_{i \leq u} \kappa_i \dim(E_i)$ where the sum is over all the positive Lyapunov exponents.  Since, for our choice of $q$, $h_\nu(T)= h_{\nu}(T \mid_{E(q)})$ is close to $\htop(T)$ and, by the Ledrappier-Young formula, each $0\leq \gamma_i(q) \leq \dim(E_i)$, we have $\gamma_i(q) \rightarrow \dim(E_i),$ as $q \rightarrow \infty$.  Again applying the Ledrappier-Young formula, we have that 
\begin{eqnarray} \label{eqnUnstableDimForq}
	\delta_u(q) = \sum_{i \leq u} \gamma_i(q) \rightarrow \dim(W^u)
\end{eqnarray}
 as $q \rightarrow \infty$.

Now $\delta_u$ is the dimension of $\nu$ on the $W^u$-manifold, which we now define following~\cite{LY1} and~\cite{LY2}.  The measure $\nu$ gives rise to conditional measures $\nu_x^u$ on $W^u(x)$ (for $\nu$-a.e. point $x$)---note that $\nu_x^u$ gives full measure to $E(q) \cap W^u(x)$.  The conditional measures allow us to define a pointwise dimension.  We first state the general definition:  the \textit{pointwise dimension of a measure $m$ at $x$} is defined to be \[\lim_{\varepsilon \rightarrow 0^+}\frac{\log m(B(x, \varepsilon))}{\log \varepsilon}\] should the limit exist.  Then Proposition~7.3.1 of~\cite{LY2} states that the pointwise dimension of $\nu_x^u$ at $x$ for $\nu$-a.e. $x$ exists and is equal to $\delta_u$.

Now, since $\nu(D(S))=1$ by Theorem~\ref{theoNuAisFullMeas}, we have that $\nu_x^u(X\setminus D(S))=0$ for $\nu$-a.e.~$x$.  Applying the mass distribution principle, following Young~\cite{Yo}, we have that 
\begin{eqnarray} \label{eqnDimonUnstable}
	\dim(D(S) \cap E(q) \cap W^u(x)) \geq \delta_u(q) \textrm{ for } \nu\textrm{-a.e. } x.
\end{eqnarray}

Since the mapping $T$ is invertible, we note that $E(q)$ is also $T^{-1}$-invariant.  Also, it is well-known that $\htop(T^{-1})=\htop(T)$ and $h_\nu(T^{-1})=h_\nu(T)$. Now, for $T^{-1}$ the unstable and stable manifolds switch.  Thus, applying the proceeding to $T^{-1}$ but keeping the notation for stable and unstable manifolds for $T$, we have that (\ref{eqnUnstableDimForq}) states that the dimension
$\delta_s(q)$ of $\nu$ along the $W^s$-foliation satisfies
\begin{eqnarray} \label{eqnStableDimForq}
	\delta_s(q) \rightarrow \dim(W^s)
\end{eqnarray}
 as $q \rightarrow \infty$.  
Let~$B_1=\{x\mid$~\eqref{eqnDimonUnstable} holds for~$x\}$ so that~$\nu(B_1)=1$. Then~$\nu_x^s(B_1\cap W^s(x))=1$ a.e.,
where~$\nu_x^s$ denotes the conditional measure on the stable manifold~$W^s(x)$. 
Arguing now in the same way as for \eqref{eqnDimonUnstable} we obtain  
\begin{eqnarray} \label{eqnDimonStable}
	\dim(B_1 \cap W^s(x)) \geq \delta_s(q) \textrm{ for } \nu\textrm{-a.e. } x.
\end{eqnarray}
%We let $A\subset E(q)$ be the set of points where~\eqref{eqnDimonUnstable}
%holds. 

To obtain the Hausdorff dimension on $X$, we use the Marstrand Slicing Theorem~\cite{Mar}, which we quote from~\cite{Kl2}:

\begin{lemm}[Marstrand Slicing Theorem]\label{lemmMarstrandSlicingTheorem} Let $M_1$ and $M_2$ be Riemannian manifolds, $A_1 \subset M_1$, $B \subset M_1 \times M_2$.  Denote by $B_a$ the intersection of $B$ with $\{a\} \times M_2$ (the slice of $B$ at an element $a$ of $A_1$) and assume that $B_a$ is nonempty for all $a \in A_1$.  Then \[\dim(B) \geq \dim(A_1) + \inf_{a \in A_1} \dim(B_a).\]
\end{lemm}

Lifting the measure~$\nu$ on~$\TT^d$ to a~$\ZZ^d$-invariant measure on~$\RR^d$ we obtain a measure on the product of the stable and unstable subspaces for $T$ on $\RR^d$. We let~$M_1$ be the stable subspace (i.e.~the sum of the contracted generalized eigenspaces for~$T$) and let~$M_2$ be the unstable subspace. Take one point $x\in X$ that satisfies \eqref{eqnDimonStable}, and let $A_1$ be the preimage of $B_1 \cap W^s(x)$ on one coset of the stable subspace in $\RR^d$.
Applying the Marstrand Slicing theorem with \eqref{eqnDimonUnstable} and \eqref{eqnDimonStable} we obtain 
\begin{eqnarray}\label{eqnDimonBoth}
	\dim(D(S) \cap E(q)) \geq \delta_u(q) + \delta_s(q).
\end{eqnarray}
  Consequently, we have that \[\dim\bigl(D(S) \cap \bigcup_{q} E(q)\bigr) \geq \dim(W^u) + \dim(W^s) = \dim(X)\] by applying (\ref{eqnUnstableDimForq}) and (\ref{eqnStableDimForq}), which proves the desired result.

\subsubsection{Toral endomorphisms}\label{subsubsecToralEndoAdjugate}
The Ledrappier-Young formula is for homeomorphisms.  There is, however, a generalization for endomorphisms, namely~\cite[Theorem~2.7 and Proposition~2.5]{QX}.  Applying this generalization to our endomorphism $T$ and following the proof for automorphisms yields (\ref{eqnUnstableDimForq}) and (\ref{eqnDimonUnstable}). As $q$ is arbitrary, the desired result follows.

%To obtain the analogous result for the stable manifold, one needs a substitute for the inverse mapping.  That substitute is the {\em adjugate (or classical adjoint)} matrix of $T$, which is the transpose of the matrix of cofactors of $T$.  (For an introduction, see \cite[Chapter 3]{Se} for example.)  Let us denote the adjugate by $\adj (T)$.  The following are properties of the adjugate:  
%
%\begin{itemize}
%\item $T  \adj (T) = \adj (T) \ T= \det (T) \  \Id_d$.
%\item If $\det(T) \neq 0$, then the eigenvalues of $\frac 1 {\det{T}} \adj(T)$ are the reciprocals of the eigenvalues of $T$.
%\end{itemize}
%
%We note that the adjugate of our matrix $T$ is also an integer-valued matrix and thus induces a valid toral endomorphism.  Now the stable manifold for $T$ corresponds to the eigenvalues of $T$ strictly less than $1$ in absolute value.  Therefore, this manifold is exactly the manifold corresponding to the eigenvalues of $\adj(T)$ strictly larger than $|\det(T)|$ in absolute value.  Consequently, this manifold is a submanifold of the unstable manifold for $\adj(T)$.  Applying the formula for endomorphisms to the submanifold, one obtains (\ref{eqnDimonStable}).  The rest of the proof is exactly the same as in the automorphism case.
%
%*******************
%
%*******************
%
%For the noninvertible case, we need to replace $T^{-1}$ with the adjugate matrix to $T$.  We then need to consider Lyapunov exponents $> \log|\det(T)|$.  Ledrappier-Young can be altered for expanding maps or it seems one can use a related result~\cite{QX}.  
%
%******************
%
%******************

\subsection{Proof of Corollary~\ref{coroCommToralMapsLarNondenDen}} \label{subsecToralProofManyMaps} 
As the automorphism group of~$\TT^d$ equals $\operatorname{GL}(d,\ZZ)$, we see that the collection of maps appearing in the corollary is countable. Each map $S_k$ in the collection of maps has a corresponding set $D(S_k)$.  Theorem~\ref{theoNuAisFullMeas} implies that \[\nu\bigl(\bigcap_k D(S_k)\bigr)=1.\]  Replacing $D(S)$ with $\bigcap_k D(S_k)$ in Section~\ref{subsecProofofMainThmforTorus} proves the desired result.

\subsection{Proof of Theorem~\ref{thmCommGroupMapsLarNondenDen}} 

 There are three ingredients:  Theorem~\ref{theoNuAisFullMeas}, the Ledrappier-Young formula, and a double application of the Marstrand slicing theorem.  We recall the notation from Section~\ref{subsecLieGroupCase}.  Summing the root spaces for which $|\lambda|>1$ for our fixed element corresponding to $a_1$ yields $\h^+$, and summing the  root spaces for which $|\lambda|<1$ for our fixed element yields $\h^-$.  These subalgebras correspond to unipotent subgroups, namely \[H^+ = \exp(\h^+) \textrm{ and } H^- = \exp(\h^-).\]  And $x H^+$ (i.e. the orbit of $x \in X$ under the right action of the subgroup $H^+$) is the leaf of the unstable manifold through $x \in X$; likewise, $x H^-$ is the   stable manifold through $x$.  Furthermore, the eigenspace $\h^0$ of the adjoint action of the element $a_1$ for the eigenvalue $1$ is the Lie algebra of the subgroup $C_G(a_1)$. As the three Lie algebras are transversal the subgroups \[H^+, H^-, \textrm{ and } H^0=C_G(a_1)\] can be used to define a local coordinate system in $G$. In fact, $H^-H^+H^0\subset G$ contains an open neighborhood of the identity of $G$.

We now  argue just as in the proof for toral automorphisms in Section~\ref{subsecProofofMainThmforTorus}. By the Ledrappier-Young formula we know that, for any measurable $A$ with $\nu(A)=1$, the dimension of $A\cap  xH^+$ and $A\cap xH^-$ are close to $\dim H^+$ and $\dim H^-$ for $\nu$-a.e. $x$. Similar to the torus case we wish to iterate this statement together with the Marstrand Slicing Theorem. However, in order to obtain the desired conclusion, we also have to consider the centralizer directions $H^0$, which we do so by using the following argument similar to that of~\cite[Section~1.5]{KM2}.

Let $h\in H^0$ be arbitrary. As $h$ commutes with $a_1$ it is clear that the push-forward $\nu h$ of $\nu$
under multiplication by $h$ on the right is a $\theta_1$-invariant probability measure with the same entropy as for $\nu$ and is supported on $E(q)h\subset ND(\theta_1)$. Applying Theorem \ref{theoNuAisFullMeas} to $\nu h$ we obtain
$(\nu h)(D(\theta_2))=1$. Equivalently we have shown  for every $h\in H^0$ and $\nu$-a.e. $x$ that $xh\in ND(\theta_1)\cap D(\theta_2)$. Applying Fubini there exists a subset $A_1$ of full $\nu$-measure such that for all $x\in A_1$ and for a.e.~$h\in H^0$ (we use the Haar measure on $H^0$) we have $xh\in ND(\theta_1)\cap D(\theta_2)$ and in particular \[\dim(xH^0\cap ND(\theta_1)\cap D(\theta_2))=\dim H^0.\] 

By the same argument using the Ledrappier-Young formula as in the torus case we obtain now that there exists a set $A_2$ of full $\nu$-measure such that $\dim(y H^+\cap A_1)\geq \delta_u(q)$ for all $y\in A_2$. Also by the same argument we find a set $A_3$ of full $\nu$-measure such that $\dim (z H^-\cap A_2)\geq \delta_s(q)$ for $z\in A_3$. Now choose and fix some $z\in A_3$. Use the Marstrand slicing theorem with $M_1=H^-$, the set $\{h^-\in H^-\mid zh^-\in A_2\}$, and $M_2=H^+$ to obtain $\dim(zH^-H^+\cap A_1)\geq \delta_s(q)+\delta_u(q)$. Another application of the Marstrand slicing theorem then gives \[\dim(ND(\theta_1)\cap D(\theta_2))\geq\delta_s(q)+\delta_u(q)+\dim(H^0).\] As $q$ is arbitrary, the theorem follows.
  
\subsection{Proof of Corollary~\ref{coroCommGroupMapsLarNondenDen}}  The proof is analogous to the proof in Section~\ref{subsecToralProofManyMaps}.

\end{document}